\documentclass[journal]{IEEEtran}

\usepackage[utf8]{inputenc}

\usepackage{amssymb,amsmath,amsthm}
\usepackage{url}

\usepackage{enumerate}

\theoremstyle{plain}
\newtheorem{theorem}{Theorem}
\newtheorem{lemma}[theorem]{Lemma}

\newtheorem{corollary}[theorem]{Corollary}
\theoremstyle{definition} 
\newtheorem{example}[theorem]{Example}

\newtheorem{definition}[theorem]{Definition}

\newcounter{claim}
\renewcommand{\theclaim}{\Alph{claim}}
\newenvironment{claim}{\refstepcounter{claim}%
\par\medskip\par\noindent{\it Claim~\theclaim.~}~\rm}%
{\par\smallskip\par}
\newenvironment{subproof}{\par\noindent{\sl Proof of Claim~\theclaim.~}}%
{$\,\triangleleft$\par\medskip\par}

\makeatletter
\def\@gifnextchar#1#2#3{\let\@tempe#1\def\@tempa{#2}\def\@tempb{#3}%
  \futurelet\@tempc\@gifnch}

\def\@gifnch{\ifx\@tempc\@sptoken\let\@tempd\@tempb%
  \else\ifx\@tempc\@tempe\let\@tempd\@tempa\else\let\@tempd\@tempb\fi\fi\@tempd}

\def\SK@set#1{\left\{#1\right\}}
\def\SK@@set#1#2{\{#1\,:\,
    \begin{array}{@{}l@{}}#2\end{array}
\}}
\def\SK@mset#1{\left\{\!\!\left\{#1\right\}\!\!\right\}}
\def\SK@@mset#1#2{\{\!\!\{#1\,:\,
    \begin{array}{@{}l@{}}#2\end{array}
\}\!\!\}}
\def\BIG@set#1{\Big\{#1\Big\}}
\def\BIG@@set#1#2{\Big\{#1\:\Big|\:
    \begin{array}{@{}l@{}}#2\end{array}
\Big\}}
\newcommand{\Set}[1]{\@gifnextchar\bgroup{\SK@@set{#1}}{\SK@set{#1}}}
\newcommand{\Mset}[1]{\@gifnextchar\bgroup{\SK@@mset{#1}}{\SK@mset{#1}}}
\newcommand{\Bigset}[1]{\@gifnextchar\bgroup{\BIG@@set{#1}}{\BIG@set{#1}}}
\makeatother

\newcommand{\refeq}[1]{(\ref{eq:#1})}
\newcommand{\of}[1]{\left( #1 \right)}
\newcommand{\ofbroken}[1]{\biggl( #1 \biggr)}
\newcommand{\function}[2]{:#1 \rightarrow #2}

\newcommand{\dd}{\!\;\mathrm{d}}

\newcommand{\mean}{\mathsf{E}}
\newcommand{\prob}[1]{\mathsf{P}[ #1 ]}

\newcommand{\sclo}[1]{\overline{#1}}

\begin{document}
\title{New bounds for the optimal density of covering single-insertion codes via the Tur\'an density}
\author{Oleg Pikhurko,
  Oleg Verbitsky, and Maksim Zhukovskii%
  \thanks{This article has been accepted for publication in \emph{IEEE Transactions on Information Theory},
    DOI:10.1109/TIT.2025.3557393.
    The work of Oleg Pikhurko was supported 
by ERC Advanced Grant 101020255. The work of Oleg Verbitsky was supported by DFG grant KO 1053/8--2.}%
\thanks{Oleg Pikhurko is with Mathematics Institute and DIMAP,
University of Warwick, Coventry CV4 7AL, UK.}%
\thanks{Oleg Verbitsky is with Institut f\"ur Informatik,
  Humboldt-Universit\"at zu Berlin, Unter den Linden 6, D-10099 Berlin, on leave from the IAPMM, Lviv, Ukraine.}%
 \thanks{Maksim Zhukovskii is with School of Computer Science, University of Sheffield S1 4DP, UK.}%
}%

\maketitle

\begin{abstract}
  We prove that the density of any covering single-insertion code $C\subseteq X^r$
  over the $n$-symbol alphabet $X$ cannot be smaller than $1/r+\delta_r$
  for some positive real $\delta_r$ not depending on $n$. This improves the volume lower
  bound of $1/(r+1)$. 
  On the other hand, we observe that, for all sufficiently large $r$, if $n$ tends to infinity
  then the asymptotic upper bound of $7/(r+1)$ due to Lenz et al.~(2021) can be improved to $4.911/(r+1)$. 

Both the lower and the upper bounds are achieved by relating the code density to the Tur\'an density from extremal combinatorics. For the last task, we use the analytic framework of measurable subsets of the real cube $[0,1]^r$.
\end{abstract}

\section{Introduction}

Let $r<k$ be positive integers and $X$ be a (not necessarily finite) set.
We say that a sequence $x\in X^r$ \emph{covers} a sequence $a\in X^k$
if $x$ is a subsequence of $a$, i.e., if $x$ is obtainable by removing $k-r$
elements from $a$ (while keeping the ordering of the remaining elements).
We say that a set $C\subseteq X^r$ \emph{covers} a set $A\subseteq X^k$ if
every sequence in $A$ is covered by at least one sequence in~$C$.

\begin{definition}\label{def:}
A set $C\subseteq X^r$ covering $X^k$ is called a \emph{covering $(k-r)$-insertion code}
over $X$. If $X=[n]$, where we denote $[n]=\{0,1,\ldots,n-1\}$,
we speak of a covering code over the $n$-symbol alphabet.
The minimum possible cardinality of such a code will be denoted by $S(n,k,r)$.
\end{definition}

\begin{example}[Grozea \cite{Grozea24}]
 $S(3,4,3)=12$ and the unique, up to renaming the symbols, optimal code
 consists of the sequences $(0,0,0)$, $(1,1,1)$, $(2,2,2)$, $(0,0,1)$, $(0,1,0)$, $(1,0,0)$, 
 $(1,1,2)$, $(1,2,1)$, $(2,1,1)$, $(2,2,0)$, $(2,0,2)$, $(0,2,2)$.
\end{example}

It is not hard to show (see Section \ref{s:prel}) that for each $k$ and $r$,
the optimal density $S(n,k,r)/n^r$ of a code converges to a limit $s(k,r)$
as $n$ increases and that
\begin{equation}
  \label{eq:S-s}
S(n,k,r)/n^r\ge s(k,r)  
\end{equation}
for all $n$. This motivates estimating the limit value $s(k,r)$,
especially because determining the exact values of $S(n,k,r)$
is computationally infeasible even for relatively small parameters $n$, $k$, and $r$ (cf.\ \cite{Grozea24} where the exact values of $S(n,4,3)$ are 
determined for $n\le5$).

We are especially interested in single-insertion codes, that is, in the case of $k=r+1$.
As observed by various researchers (e.g.~\cite[Eq.~(2)]{Levenshtein66} for $n=2$ and \cite[Lemma 4.1]{Viennot83}
for general $n$), every sequence $x\in[n]^r$ covers exactly
$(r+1)(n-1)+1$ sequences in $[n]^{r+1}$, which immediately yields
$$
S(n,r+1,r)\ge\frac{n^{r+1}}{(r+1)(n-1)+1}.  
$$
On the other hand, Lenz et al.~\cite{LenzRSY21} proved that
$$
S(n,r+1,r)\le\frac{7n^{r+1}}{(r+1)(n-1)+1}.  
$$
These estimates readily imply that
\begin{equation}
  \label{eq:Lev-Lenz}
  \frac1{r+1}\le s(r+1,r)\le\frac7{r+1}.
\end{equation}
In the present paper, we aim at improving the lower and the upper bound in~\refeq{Lev-Lenz}.

We begin with addressing the question whether or not the lower bound in \refeq{Lev-Lenz}
is sharp. The equality $s(r+1,r)=1/(r+1)$ would mean the existence of asymptotically
perfect covering single-insertion codes. Our first estimate rules out this possibility by
showing that
\begin{equation}
  \label{eq:1/r}
 s(r+1,r)\ge\frac1r. 
\end{equation}

We prove the lower bound \refeq{1/r} in a natural analytic framework of measurable
covering single-insertion codes over the real segment $[0,1]$.
This framework is useful for establishing a relationship between covering codes
and Tur\'an systems, the classical and actively studied subject in combinatorics \cite{Keevash11,Ruszinko07,Sidorenko95}.
Of crucial importance for us is the concept of the extremal Tur\'an density $t(k,r)$
(see Section \ref{s:Turan} for the definition). We notice that
\begin{equation}
  \label{eq:s-t}
  s(k,r)\le t(k,r)
\end{equation}
and, therefore, any upper bound for $t(k,r)$ yields also an upper bound for $s(k,r)$.
The currently best upper bounds for the Tur\'an density $t(r+1,r)$ have recently been obtained
by Pikhurko \cite{Pikhurko25} who proved that $t(r+1,r)\le6.239/(r+1)$ for all $r$
and $t(r+1,r)\le4.911/(r+1)$ for all sufficiently large $r$. Both of these bounds
imply an improvement of the upper bound in~\refeq{Lev-Lenz}. 
These improvements can also be derived by a more careful analysis of the construction in \cite{LenzRSY21}.\footnote{Using the same argument as in \cite[Lemma~2.3]{Pikhurko25},
it can be showed that \cite[Lemma 8]{LenzRSY21} in fact holds if $\mu_I\le 4.911$ and $q$ is large enough.}
In any case, it is remarkable that the state-of-the-art upper bounds for $s(r+1,r)$
are actually provided by the available upper bounds for~$t(r+1,r)$.

Somewhat surprisingly, we obtain a relation between $s(r+1,r)$ and $t(r+1,r)$
also in the other direction: Any lower bound for $t(r+1,r)$ better than $1/r$
implies a lower bound for $s(r+1,r)$ also better than $1/r$.
Lower bounds $t(r+1,r)\ge1/r+\epsilon_r$ for $\epsilon_r>0$ are obtained
by Chung and Lu \cite{ChungL99} and Lu and Zhao \cite{LuZ09}
and, therefore, our initial lower bound \refeq{1/r} can be further improved
to $s(r+1,r)\ge1/r+\delta_r$ for some $\delta_r>0$. Though we provide
explicit values of $\delta_r$ in the main body of the paper,
right now we prefer to summarize the new bounds for the optimal density of covering single-insertion
codes, improving the current bounds \refeq{Lev-Lenz}, in a somewhat simplified form.

\begin{theorem}\label{thm:intro}
  For $s(k,r)=\lim_{n\to\infty}S(n,k,r)/n^r$,
  $$
\frac1r<s(r+1,r)\le\frac{4.911}{r+1}
$$
where the former inequality is true for all $r$ and the latter inequality is true for all sufficiently large~$r$.
\end{theorem}
Taking into account the inequality \refeq{S-s}, note that the lower bound stated in
Theorem \ref{thm:intro} is not just asymptotic, as it yields a lower bound
$S(n,k,r)/n^r\ge1/r+\delta_r$ for some real $\delta_r$ not depending on $n$
(an explicit value of $\delta_r$ will be specified in the sequel).
Covering codes over large alphabets naturally arise in research driven by applications in computational biology and genomics. Notably, DNA and RNA sequences are constructed from five canonical nucleobases: A, C, G, T, and U. Furthermore, the genetic code of life involves 22 proteinogenic amino acids, while over 500 amino acids are known to occur in nature. For results concerning the Hamming metric, we refer to \cite{KeriO05}, while results for the Levenshtein metric (which can be described using insertion/deletion codes) can be found in \cite{BhardwajPRS21}. Covering insertion/deletion codes over arbitrarily large alphabets have also been studied in \cite{AfratiSMPU12,AfratiSRRSU14} in the context of the MapReduce framework for data analytics.

The paper is organized as follows. The convergence of the optimal code density $S(n,k,r)/n^r$
to a limit $s(k,r)$ is showed in Section \ref{s:prel}. An analytic framework for
estimation of $s(k,r)$ is suggested in Section \ref{s:analytic}.
Our first lower bound \refeq{1/r} is established in Section \ref{s:lower}.
In Section \ref{s:Turan} we introduce Tur\'an systems and prove the relation \refeq{s-t},
thereby obtaining the upper bound in Theorem \ref{thm:intro} (restated as Corollary \ref{cor:OP}).
A reverse relation between $s(r+1,r)$ and $t(r+1,r)$ is proved in Section \ref{s:further}
as Theorem \ref{thm:tr-sr} and Corollary \ref{cor:s-t-lower}, which allows us to
improve \refeq{1/r} to a strict inequality stated in Theorem \ref{thm:intro} in a simplified form
and made more precise in Corollaries \ref{cor:odd} and~\ref{cor:s43}.
Note that the proof of Theorem \ref{thm:tr-sr} is heavily based on the argument
used in Section \ref{s:lower} for obtaining the bound in~\refeq{1/r}.

\section{Preliminary lemmas}\label{s:prel}

Given a function $f\function YX$, we define a function $f^r\function{Y^r}{X^r}$
by $f^r(y_1,\ldots,y_r)=(f(y_1),\ldots,f(y_r))$. The preimage of a set $C\subseteq X^r$ under $f^r$
will be denoted by $f^{-r}(C)$. 

\begin{lemma}\label{lem:f}
If $C\subseteq X^r$ covers $X^k$ and $f$ is an arbitrary function from $Y$ to $X$, 
then $f^{-r}(C)$ covers $Y^k$.  
\end{lemma}

\begin{proof}
Consider an arbitrary $(y_1,\ldots,y_k)\in Y^k$ and denote $(x_1,\ldots,x_k)=f^k(y_1,\ldots,y_k)$.
Since $C$ covers $X^k$, some $r$-dimensional projection of $(x_1,\ldots,x_k)$
belongs to $C$. Let, say, $(x_1,\ldots,x_r)\in C$. It remains to note that $(y_1,\ldots,y_k)$
is covered by the vector $(y_1,\ldots,y_r)$ in $f^{-r}(C)$.
\end{proof}

\begin{lemma}\label{lem:s}
  For all positive integers $k>r$, the optimal code density $S(n,k,r)/n^r$ converges to a limit $s(k,r)$
  as $n$ grows, and $S(n,k,r)/n^r\ge s(k,r)$ for all~$n$.
\end{lemma}

\begin{proof}
Let $C\subseteq [n]^r$ be an optimal covering $(k-r)$-insertion code, that is, $|C|=S(n,k,r)$.
Let $m>n$ and define $f\function{[m]}{[n]}$ by $f(y)=y\bmod n$ for all $y\in[m]$.
By Lemma \ref{lem:f}, the preimage $f^{-r}(C)$ is a covering $(k-r)$-insertion code
over $[m]$. Let $q=\lfloor m/n \rfloor$. Thus,
$$
S(m,k,r)\le|f^{-r}(C)|\le(q+1)^r\,|C|.
$$
It follows that
$$
\limsup_{m\to\infty}\frac{S(m,k,r)}{m^r}\le\frac{S(n,k,r)}{n^r},
$$
implying both statements in the lemma.
\end{proof}

\section{Analytic reformulation}\label{s:analytic}

Definition \ref{def:} admits consideration of an infinitary setting.
In order to be able to speak about the size of a code, we suppose that
$X$ is a measurable space endowed with a probability measure $\lambda$.
The Cartesian power $X^r$ is endowed with the product measure,
which for brevity will be denoted also by $\lambda$. 
We define $s(X,k,r)=\inf_C\lambda(C)$ where the infimum is taken over all measurable $(k-r)$-insertion covering codes $C\subseteq X^r$.
In the discrete case, we endow $X=[n]$ with the uniform probability measure,
getting $s([n],k,r)=S(n,k,r)/{n^r}$, which is just another notation
for the optimal code density over a finite alphabet.
For the unit segment of reals $X=[0,1]$, let $\lambda$ be the Lebesgue measure.
In this case,
\begin{equation}
  \label{eq:s}
s([0,1],k,r)=\inf_{C}\lambda(C)  
\end{equation}
where the infimum is taken over Lebesgue measurable $(k-r)$-insertion covering codes $C\subseteq [0,1]^r$.

\begin{theorem}\label{thm:s-limit}
$s(k,r)=s([0,1],k,r)$.
\end{theorem}

\begin{proof}
  We first prove that $s([0,1],k,r)\le s(k,r)$.
  Let $C\subseteq[n]^r$ be an optimal covering $(k-r)$-insertion code, that is, $|C|=S(n,k,r)$.
  Define $f_n\function{[0,1]}{[n]}$ by $f_n(0)=0$ and $f_n(x)=i$ for all $x\in(\frac in,\frac{i+1}n]$.
  By Lemma \ref{lem:f}, the preimage $f_n^{-r}(C)\subseteq[0,1]^r$
  is a covering $(k-r)$-insertion code over $[0,1]$.
 Note that $f_n$ is a measurable function, and $\lambda(f_n^{-r}(C))=|C|/n^r$.
 This shows that
 $$
 s([0,1],k,r)\le\lambda\of{f_n^{-r}(C)}=\frac{|C|}{n^r}=s([n],k,r)
 $$
 for all $k>r$, implying the required inequality.

 In order to prove that $s(k,r)\le s([0,1],k,r)$, we use the following convention.
 A mapping $\tau\function{[t]}{Z}$ can be identified with the sequence $(\tau(0),\tau(1),\ldots,\tau(t-1))\in Z^t$.
 After this, it makes sense to say, for example, that $\rho\function{[r]}{Z}$
 covers $\kappa\function{[k]}{Z}$.
 
 Let $C\subseteq [0,1]^r$ be a covering $(k-r)$-insertion code over $[0,1]$.
 Given a sequence $y=(y_0,y_1,\dots,y_{n-1})$ in $[0,1]^n$, we define $C_n=C_n(y)$ as the set
 of all mappings $\rho\function{[r]}{[n]}$ such that $(y_{\rho(0)},y_{\rho(1)},\ldots,y_{\rho(r-1)})\in C$.
 According to our convention, we view $C_n$ as a subset of $[n]^r$ and claim
 that $C_n$ covers $[n]^k$, i.e., that it is a covering $(k-r)$-insertion code over $[n]$.
 Indeed, take an arbitrary mapping $\kappa\function{[k]}{[n]}$. Since $C$ covers $[0,1]^k$,
 the sequence $(y_{\kappa(0)},y_{\kappa(1)},\ldots,y_{\kappa(k-1)})$ is covered by some subsequence
 $(y_{\kappa(i_0)},\ldots,y_{\kappa(i_{r-1})})\in C$ where $0\le i_0<i_1<\dots<i_{r-1}<k$.
 Define  $\rho\function{[r]}{[n]}$ by setting $\rho(0)=\kappa(i_0),\dots,\rho(r-1)=\kappa(i_{r-1})$
 and note that $\rho\in C_n$ covers~$\kappa$.

 It follows that $S(n,k,r)\le|C_n(y)|$ for every $y\in [0,1]^n$.
 This implies that if we take $y$ uniformly at random in $[0,1]^n$ or,
 equivalently, if we take independent random variables $y_0,\dots,y_{n-1}$
 uniformly distributed in $[0,1]$, then $S(n,k,r)$ does not exceed the expectation $\mean|C_n(y)|$,
 which by linearity is equal to the sum $\sum_{\rho}\prob{(y_{\rho(0)},\ldots,y_{\rho(r-1)})\in C}$
 of probabilities over all maps $\rho\function{[r]}{[n]}$.
If $\rho$ is injective, then $(y_{\rho(0)},\ldots,y_{\rho(r-1)})\in C$ with probability $\lambda(C)$
and, therefore,
\begin{multline*}
S(n,k,r)\le n(n-1)\cdots(n-r+1)\lambda(C)\\+\of{n^r-n(n-1)\cdots(n-r+1)}.
\end{multline*}
Using Lemma \ref{lem:s}, we derive from here
$$
s(k,r)\le\frac{S(n,k,r)}{n^r}\le\lambda(C)+o(1).
$$
As $C$ can be chosen with $\lambda(C)$ arbitrarily close to $s([0,1],k,r)$,
this proves the inequality $s(k,r)\le s([0,1],k,r)$.
\end{proof}

While some of the forthcoming proofs use the analytic setting of a measurable subset $X\subseteq [0,1]^r$, they can also be easily re-written to work directly with subsets of $[n]^r$. The choice of which language to use is just the matter of convenience.

\section{A better lower bound for $s(r+1,r)$}\label{s:lower}

We now improve the lower bound in \refeq{Lev-Lenz}.

\begin{theorem}\label{thm:1/r}
$s(r+1,r)\ge1/r$.
\end{theorem}

\noindent
Recall that Theorem \ref{thm:s-limit} allows us to switch to the analytic setting, where we
have to prove that $s([0,1],r+1,r)\ge1/r$. The proof is based on two lemmas below.

Let $C_1,\ldots,C_k$ be measurable sets in a space $\Omega$ with probability measure $\lambda$.
For two indices $i$ and $j$ such that $i<j$, let $C_{i,j}=C_i\cap C_j$.
By Bonferroni's inequality,
$$
\lambda\of{\bigcup_{i=1}^kC_i}\ge\sum_{i=1}^k\lambda(C_i)-\sum_{1\le i<j\le k}\lambda(C_{i,j}).
$$
We, however, need an inequality in the opposite direction.

\begin{lemma}[An inverse Bonferroni's inequality]\label{lem:anti-bonf}
Let $T$ be a tree with vertex set $V(T)=\{1,2,\ldots,k\}$ and edge set $E(T)$. Then
\begin{equation}
  \label{eq:anti-bonf}
\lambda\of{\bigcup_{i=1}^kC_i}\le\sum_{i=1}^k\lambda(C_i)-\sum_{e\in E(T)}\lambda(C_e).  
\end{equation}
\end{lemma}

\begin{proof}
Note that $\lambda\of{\bigcup_{i=1}^kC_i}=\sum_A\lambda(A)$, where the sum is over all atomic sets $A$
in the Boolean algebra generated by the subsets $C_1,\ldots,C_k$ of $\Omega$, apart $
A=\Omega\setminus \bigcup_{i=1}^kC_i$.
Consider a particular atomic set $A$ and let $t(A)$ denote the number of sets $C_i$ including $A$ as a subset. If $t(A)>0$, then $A$ is included in at most $t(A)-1$ of the sets $C_e$ with $e\in E(T)$. This is true because a set of $t$ vertices spans the subgraph of the tree $T$ with at most $t-1$ edges. It follows that
$$\sum_{e\in E(T)} \lambda(C_e) \leq  \sum_{A:\,t(A)>0} (t(A)-1)\cdot \lambda(A).$$
We conclude that
$$
\sum_{e\in E(T)}\lambda(C_e)+\lambda\of{\bigcup_{i=1}^kC_i}\le
\sum_A  t(A) \cdot \lambda(A) = \sum_{i=1}^k\lambda(C_i),
$$
completing the proof.
\end{proof}

Let $X=[0,1]$ and $\lambda$ be the Lebesgue measure on $X$.
We write $\lambda$ also to denote the corresponding product measure on $X^{r+1}$.

Let $C\subseteq X^r$. For $i\le r+1$, we define $C_i\subseteq X^{r+1}$
as the set of all sequences $x$ in $X^{r+1}$ such that the subsequence obtained
by removing the $i$-th element of $x$ belongs to $C$.
Theorem \ref{thm:1/r} immediately follows from the next lemma.

\begin{lemma}\label{lem:1/r}
  If $C\subseteq X^r$ covers $X^{r+1}$, i.e., $\bigcup_{i=1}^{r+1}C_i=X^{r+1}$, then $\lambda(C)\ge1/r$.
\end{lemma}

\begin{proof}
Applying Lemma \ref{lem:anti-bonf} to $C_1,\dots,C_{r+1}\subseteq [0,1]^{r+1}$ for any fixed tree $T$,
we readily obtain
$$
\sum_{e\in E(T)}\lambda(C_e)\le \sum_{i=1}^{r+1} \lambda(C_i) -\lambda\left(\bigcup_{i=1}^{r+1}C_i\right)= (r+1)\lambda(C)-1.
$$
We now estimate the left-hand side from below.
Denote the characteristic function of $C$ by $\chi_C$. Using Fubini's theorem
along with the Cauchy-Bunyakovsky-Schwarz inequality, we get
\begin{multline*} 
\lambda(C_{1,2})=
\displaystyle\int_{X^{r+1}} 
\chi_C(x_1,x_3,\ldots,x_{r+1})\\
\mbox{}\hspace{25mm}
\times\chi_C(x_2,x_3,\ldots,x_{r+1}) 
\dd x_1\cdots\dd x_{r+1} \\ =
\int_{X^{r-1}}\ofbroken{
\int_{X}
\chi_C(x_1,x_3,\ldots,x_{r+1}) 
\dd x_1
\hspace{23mm}\mbox{}\\
\mbox{}\hspace{15mm}\times
\int_{X}
\chi_C(x_2,x_3,\ldots,x_{r+1})
\dd x_2
}
\dd x_3\cdots\dd x_{r+1} \\ =
\int_{X^{r-1}}\of{
\int_{X}
\chi_C(x,x_3,\ldots,x_{r+1})
\dd x
}^2
\dd x_3\cdots\dd x_{r+1}
 \\ \ge
\of{
\int_{X^{r-1}}
\int_{X}
\chi_C(x,x_3,\ldots,x_{r+1})
\dd x
\dd x_3\cdots\dd x_{r+1}
}^2 \\ =
\of{\int_{X^r}  \chi_C(x,x_3,\ldots,x_{r+1})  \dd x\dd x_3\cdots\dd x_{r+1}}^2
= \lambda(C)^2.
\end{multline*}

Each of the $r$ values $\lambda(C_e)$ for $e\in E(T)$ is estimated similarly.
It follows that
$$
r\lambda(C)^2\le(r+1)\lambda(C)-1.
$$
Rewriting this as
\begin{equation}
  \label{eq:later}
(1-\lambda(C))(r\lambda(C)-1)\ge0,  
\end{equation}
we conclude that $\lambda(C)\ge1/r$.
\end{proof}

We conclude this section with a discussion of a consequence of Theorem \ref{thm:1/r}.
We call $X\subseteq[n]^{r+1}$ a \emph{1-packing}
if no two sequences in $X$ have a common subsequence of length $r$ (or, equivalently,
if the minimum Levenshtein distance between two elements of $X$ is larger than 2).
Denote the maximum size $|X|$ of a 1-packing $X\subseteq [n]^{r+1}$ by $P(n,r+1,r)$.
The packing and covering numbers are related by the inequality
\begin{equation}\label{eq:PS}
P(n,r+1,r)\le S(n,r+1,r); 
\end{equation}
 indeed, every element of a covering $1$-insertion code $C\subseteq [n]^r$ can cover at most one element of a maximum packing $X\subseteq [n]^{r+1}$.
It is known \cite[Cor.~5.1]{Levenshtein92} that $P(n,r+1,r)/n^r\sim1/(r+1)$ if $n/r\to\infty$. 
Taking this result into account, our Theorem \ref{thm:1/r} separates the 
two values in \refeq{PS} by showing an additive gap at least $1/(r(r+1))$
between their density versions in the setting when $r$ is fixed and $n$ grows.

\section{Tur\'an systems}\label{s:Turan}

If $X\subseteq Y$ are any sets, then one can say that $Y$ covers $X$, but by a kind of duality we will also say
that \emph{$X$ covers $Y$}. Let $1<r<k<n$. A family $C$ of $r$-element subsets of $[n]$ is
called a \emph{Tur\'an $(n,k,r)$-system} if every $k$-element subset of $[n]$ is
covered by at least one member of $C$. The minimum possible cardinality of $C$
is denoted by $T(n,k,r)$. A well-known argument \cite{Keevash11,Sidorenko95} shows that $(n-r)\,T(n,k,r)\ge n\,T(n-1,k,r)$,
which implies that the densities $T(n,k,r)/{n\choose r}$ form a non-decreasing sequence
for each $k$ and $r$. The limit is called the \emph{Tur\'an density} and denoted by $t(k,r)$. For surveys including Tur\'an systems, see~\cite{Keevash11,Ruszinko07,Sidorenko95}.

We connect Tur\'an systems and covering insertion codes by showing that
the former concept can in limit be seen as a symmetric version of the latter concept.

Call a set $C\subseteq X^r$ \emph{symmetric} if $C$ is closed with respect to all
permutations of the $r$ coordinates. Let us start with a simple observation.

\begin{lemma}\label{lem:f-sym}
If $C\subseteq X^r$ is symmetric, then $f^{-r}(C)$ is also symmetric for any function $f\function YX$.  
\end{lemma}

The Tur\'an density $t(k,r)$ can be characterized in terms of an analytic
object similarly to Theorem \ref{thm:s-limit}. 
Specifically, we define $s^\star([0,1],k,r)$ similarly to \refeq{s}
with the additional condition that the infimum is taken over symmetric codes.

\begin{theorem}\label{thm:t-limit} For all positive integers $k>r$, we have
$t(k,r)=s^\star([0,1],k,r)$.
\end{theorem}

\begin{proof}
  We first prove the inequality $s^\star([0,1],k,r)\le t(k,r)$.
  For a set $A\subseteq[n]^r$, let $A^\dagger$ denote the set of all sequences in $A$ with pairwise
  distinct elements and $A^\ddagger=A\setminus A^\dagger$ be the remaining part of~$A$.
  
Let $T_n$ be an optimal Tur\'an $(n,k,r)$-system, that is, $|T_n|=T(n,k,r)$.
Convert $T_n$ into a covering $(k-r)$-insertion code $C_n\subseteq[n]^r$ as follows.
For each $r$-element set in $T_n$, place all its $r!$ orderings in $C_n$,
thereby covering all sequences in $([n]^k)^\dagger$.
In order to cover the remaining sequences, we just add $([n]^r)^\ddagger$ to $C_n$.
Note that $C_n$ is symmetric and that
\begin{multline}
|C_n|\le |T_n|\,r!+|([n]^r)^\ddagger|=|T_n|\,r!+\\\of{n^r-n(n-1)\cdots(n-r+1)}.\label{eq:Cn}
\end{multline}
Consequently,
$$
\frac{|C_n|}{n^r}\le\frac{|T_n|}{{n\choose r}}+o(1)
$$ 
where the little-o term approaches 0 as $n$ increases.
Consider $D_n=f_n^{-r}(C_n)$ for the function $f_n\function{[0,1]}{[n]}$ defined
in the proof of Theorem \ref{thm:s-limit}. Note that $f_n$ has preimages of measure $1/n$ each.
By Lemma \ref{lem:f}, $D_n\subseteq[0,1]^r$ is a covering
$(k-r)$-insertion code over $[0,1]$. Since $C_n$ is symmetric,
$D_n$ is also symmetric by Lemma \ref{lem:f-sym}.
It follows that
\begin{multline*}
s^\star([0,1],k,r)\le\lambda(D_n)=\frac{|C_n|}{n^r}\\
\le\frac{|T(n,k,r)|}{{n\choose r}}+o(1)
\le t(k,r)+o(1),
\end{multline*}
yielding the required inequality.

We now prove that, conversely, $t(k,r)\le s^\star([0,1],k,r)$.
 Let $C\subseteq [0,1]^r$ be an arbitrary measurable symmetric covering $(k-r)$-insertion code over $[0,1]$. Given an integer $n\ge k$, consider a sequence $y=(y_0,\dots,y_{n-1})$ in $[0,1]^n$ with pairwise different elements.
 Define a family $G=G(y)$ of $r$-element subsets of $[n]$ by putting $\{i_1,\dots,i_r\}\subseteq [n]$ in $G$ if and only if $(y_{i_1},\dots,y_{i_r})\in C$. The last condition does not depend on the order of indices by the symmetry of~$C$. 

Let us show that $G$ is a Tur\'an $(n,k,r)$-system. Take any $k$-element set $K\subseteq [n]$. Since the subsequence $(y_i)_{i\in K}$ of $y$ is covered by some sequence in $C$, there is an
$r$-element set $\{i_1,\dots,i_r\}\subseteq K$ such that $(y_{i_1},\dots,y_{i_r})\in C$. By definition, this means that $\{i_1,\dots,i_r\}\in G$. Since $K$ was an arbitrary $k$-element subset of $[n]$, $G$ is indeed a Tur\'an $(n,k,r)$-system.

Now, take a uniformly random $y=(y_0,\dots,y_{n-1})$ in $[0,1]^n$; equivalently, we take independent uniform $y_0,\dots,y_{n-1}\in [0,1]$. With probability 1, all $y_i$ are different. To compute the expected number of $r$-element sets belonging to $G$, we sum the probability that $R\in G$ over all $R=\{i_1,\dots,i_r\}\subseteq [n]$. By the uniformity of $(y_0,\dots,y_{n-1})\in [0,1]^n$, we have that $(y_{i_1},\dots,y_{i_r})$ is a uniform element of $[0,1]^r$. Thus, the probability that $(y_{i_1},\dots,y_{i_r})\in C$ (which is exactly the  probability that $R\in G$) is equal to the measure $\lambda(C)$ of $C$. We conclude that $\mean |G|={n\choose r}\lambda(C)$.

Of course, if we remove from $[0,1]^n$ the null-set $D$ of points $y$ where some two coordinates $y_i$'s coincide, then the expectation does not change.
Take $(y_0,\dots,y_{n-1})\in [0,1]^n\setminus D$ such that $|G|$ is at most its expected value ${n\choose r}\lambda(C)$. Then the density of $G$ is most $\lambda(C)$. Since $n$ and $C$ were arbitrary, with $\lambda(C)$ arbitrarily close to $s^\star([0,1],k,r)$, the required inequality follows.
\end{proof}

One can show that the appropriately defined parameter $s^\star(X,k,r)$ (resp.\ $s(X,k,r)$) is the same for all atomless probability spaces $X$, since each such space admits, for every $n$, a measurable partition into parts of measure $1/n$ each.  

For fixed $k$ and $r$, one can alternatively define the function $s^\star(k,r)$ using the \emph{$r$-hypergraphon} limit object introduced by Elek and Szegedy~\cite{ElekSzegedy12}. While the advantage of this approach is that the infimum in the definition would be in fact the minimum (that is, would be attained) potentially allowing for further methods like variational calculus, the limit object is rather complicated and requires a lot of technical preliminaries. So we stay with our simple setting of measurable subsets of $[0,1]^r$.

We state an immediate consequence of~\eqref{eq:Cn} (which also follows from Theorems \ref{thm:s-limit} and~\ref{thm:t-limit}).

\begin{corollary}\label{cor:s-t}
  $s(k,r)\le t(k,r)$.
\end{corollary}

Theorem \ref{thm:1/r}, therefore, implies that
$$
t(r+1,r)\ge s(r+1,r)\ge1/r.
$$
This lower bound $t(r+1,r)\ge1/r$ was shown  independently by de Caen~\cite{Decaen83ac}, Sidorenko~\cite{Sidorenko82}, and Tazawa and Shirakura~\cite{TazawaShirakura83}
and generalized by de Caen~\cite{Decaen83}.
Thus, Theorem \ref{thm:1/r} is an extension of this classical result to
the realm of covering insertion codes.

On the other hand, no analogue of the upper bound $s(r+1,r)=O(1/r)$ (see \refeq{Lev-Lenz})
was known for $t(r+1,r)$. Quite the contrary, de Caen~\cite{Decaen94} conjectured that
$r\cdot t(r+1,r)\to\infty$ as $r$ grows. Inspired by the relationship between Tur\'an systems
and covering insertion codes, which we pinpoint here, Pikhurko \cite{Pikhurko25}
disproved this conjecture by showing that $t(r+1,r)\le6.239/(r+1)$ for all $r$
and $t(r+1,r)\le4.911/(r+1)$ for all sufficiently large $r$. By Corollary \ref{cor:s-t},
the same upper bounds apply to~$s(r+1,r)$. Alternatively, Corollary~\ref{cor:OP} below follows from the recurrence in~\cite{LenzRSY21}, via the same analysis as that in the proof of \cite[Lemma~2.3]{Pikhurko25}.

\begin{corollary}\label{cor:OP}
  \hfill

  \begin{enumerate}[\bf 1.]
  \item
    $s(r+1,r)\le6.239/(r+1)$ for all $r$.
  \item
    $s(r+1,r)\le4.911/(r+1)$ for all sufficiently large~$r$.
  \end{enumerate}
\end{corollary}

In the particular case of $r=2$, we have
$$
s([0,1],3,2)=s^\star([0,1],3,2)=1/2.
$$
Indeed, $s([0,1],3,2)\ge1/2$ by Theorem \ref{thm:1/r}.
The upper bound $s^\star([0,1],3,2)\le1/2$ is provided
by the symmetric single-insertion code $[0,\frac12]^2\cup(\frac12,1]^2$ covering
the cube $[0,1]^3$,
which is an analog of the single-insertion code $\Set{(0,0),(1,1)}$
covering the Boolean cube $\{0,1\}^3$.

Along with Theorem \ref{thm:t-limit}, this implies that
$
t(3,2)=\frac12,
$
which is a well-known fact belonging to the basics of graph theory.
The lower bound $t(3,2)\ge\frac12$ is known as Mantel's theorem.
The upper bound $t(3,2)\le\frac12$ follows
by considering the disjoint union of complete graphs $K_{\lfloor n/2\rfloor}$ and~$K_{\lceil n/2\rceil}$.

We conclude this section with an overview of the known bounds on $t(r+1,r)$ for $r\ge3$.

\paragraph{Bounds for small $r$}

For $r=3$ it is known that
\begin{equation}
  \label{eq:r=3}
0.438334\le t(4,3)\le\frac49=0.444\ldots.  
\end{equation}
The lower bound is due to Razborov \cite{Razborov10}.
The upper bound, conjectured to be optimal, is given by many different constructions, one of which is the following.
Split $[n]$ into three parts $V_0,V_1,V_2$ as evenly as possible and put a 3-element set in $C$
if it either lies entirely inside some $V_i$ or, for some residues $i$ modulo $3$, has two elements in $V_i$ and one element in~$V_{i+1}$.

An account of the known bounds on $t(r+1,r)$ for other small values of $r$ can
be found in the survey~\cite{Sidorenko95}.

\paragraph{General bounds}

The bound $t(r+1,r)\ge1/r$ is improved in \cite{ChungL99}
for odd $r$ and in \cite{LuZ09} for even $r$.
For all odd $r\ge3$, it is shown in \cite{ChungL99} that
\begin{equation}
  \label{eq:r-odd}
 t(r+1,r)\ge\frac{5r-\sqrt{9r^2+24r}+12}{2r(r+3)}=\frac1r+\frac1{r^2}+O(r^{-3}).
\end{equation}
For all even $r\ge4$, it is shown in \cite{LuZ09} that
\begin{equation}
  \label{eq:r-even}
 t(r+1,r)\ge\frac1r+\frac{(1-1/r^{p-1})(r-1)^2}{2r^p\of{{r+p\choose p-1}+{r+1\choose2}}},
\end{equation}
where $p$ is the least prime factor of $r-1$. This bound is the strongest if $p=3$, that is,
$r=4\pmod6$. In this case, it reads
$$
 t(r+1,r)\ge\frac1r+\frac1{2r^3}+O(r^{-4}).
$$
In the worst case, which happens when $p=r-1$, Bound \refeq{r-even} yields
$$
 t(r+1,r)\ge\frac1r+\frac{1-o(1)}{4r^{r-3}{2r\choose r}}.
$$

\section{A further improvement of the lower bound for $s(r+1,r)$}\label{s:further}

We now improve Theorem \ref{thm:1/r} by showing that any lower bound for $t(r+1,r)$
better than $1/r$ implies a lower bound for $s(r+1,r)$ better than~$1/r$.

We have $s(r+1,r)\le t(r+1,r)$ by Corollary \ref{cor:s-t}.
Let us prove a relation in the opposite direction.

\begin{theorem}\label{thm:tr-sr} For every $r\ge 3$, it holds that
  \begin{multline*}
  t(r+1,r)\le s(r+1,r)\\
  +2r!\sqrt{r(r+1)(1-s(r+1,r))\of{s(r+1,r)-\frac1r}}.
  \end{multline*}
\end{theorem}

\begin{proof}
  Let $X=[0,1]$ and let $\lambda$ denote the Lebesgue measure on $X$.
  Moreover, we write $\lambda$ to denote also the corresponding product measure on any $k$-dimensional cube $X^k$.
Consider a measurable set $C\subseteq X^r$. As in Section \ref{s:lower}, for each $i\le r+1$ we define $C_i\subseteq X^{r+1}$
to be the set of all sequences $x$ in $X^{r+1}$ such that the subsequence obtained
by removing the $i$-th element of $x$ belongs to $C$. Suppose that $C$ covers $X^{r+1}$, that is,
$X^{r+1}=\bigcup_{i=1}^{r+1}C_i$. We define
\begin{eqnarray*}
K&=&\bigcap_{i=1}^{r+1}C_i,\\
P_i&=&C_i\setminus\bigcup_{j\ne i}C_j,\\
R&=&X^{r+1}\setminus\of{K\cup\bigcup_{i=1}^{r+1}P_i},\\
R_i&=&R\setminus C_i.  
\end{eqnarray*}
In other words, $K$ is
the atomic set of the Boolean algebra generated by $C_1,\dots,C_{r+1}$ occurring in these
sets with the maximum possible multiplicity $t(K)=r+1$. For each $i\le r+1$, $P_i$ is
an atomic set of minimum possible multiplicity $t(P_i)=1$ (as 0 is impossible by the covering property). The remaining part $R$
is the union of all atomic sets $A$ of intermediate multiplicity $1<t(A)\le r$.
Finally, $R_i$ is the part of $R$ formed by the atomic sets outside~$C_i$.

Let $k=r+1$.
The inverse Bonferroni's inequality given by Lemma \ref{lem:anti-bonf}
can be somewhat improved. While $T$ was in this lemma an arbitrary tree
on vertices $1,\dots,k$, let $T_j$ be now the star with centre at $j$, that is,
$E(T_j)=\{\{j,i\}: 1\le i\le r+1,\ i\not=j\}$. In the case that $T=T_j$,
Inequality \refeq{anti-bonf} can be improved to
$$
\lambda\of{\bigcup_{i=1}^kC_i}\le\sum_{i=1}^k\lambda(C_i)-\sum_{e\in E(T_j)}\lambda(C_e)-\lambda(R_j),
$$
 which can be routinely verified by looking at the contribution of each atomic set.
Since $C$ covers $X^{r+1}$, the left hand side is equal to 1.
Arguing as in the proof of Lemma \ref{lem:1/r}, in place of Inequality \refeq{later} we obtain
$$
\lambda(R_j)\le(1-\lambda(C))(r\lambda(C)-1)
$$
for each $j\le r+1$. It follows that
\begin{multline}
\lambda(R)\le\sum_{j=1}^{r+1}\lambda(R_j)\le(r+1)(1-\lambda(C))(r\lambda(C)-1)\\
  = r(r+1)(1-\lambda(C))\of{\lambda(C)-\frac1r}.\label{eq:lambdaR}
\end{multline}
This shows that if the covering code $C$ has density sufficiently close to $1/r$, then up to
a small set $R$, the Boolean algebra generated by $C_1,\dots,C_{r+1}$
is the sunflower with kernel $K$ and petals $P_1,\dots,P_{r+1}$.

Given a set $M\subseteq X^k$, we define its \emph{symmetric closure} $\sclo M$
to be the inclusion-minimal symmetric superset of~$M$:
\begin{multline*}
 \sclo M=\{(x_1,\dots,x_k)\in X^k: \exists \mbox{ permutation $\sigma$ of $\{1,\dots, k\}$}\\
 \mbox{with $(x_{\sigma(1)},\dots,x_{\sigma(k)})\in M$}\}.
\end{multline*}

\begin{claim}\label{cl:K-sym}
  $K\setminus\sclo R$ is symmetric.
\end{claim}

\begin{subproof}
  Let $(x_1,x_2,x_3,\dots,x_{r+1})\in K\setminus\sclo R$.
  Since $(x_1,x_2,x_3,\dots,x_{r+1})\in K\subseteq C_1\cap C_2$, we have
  $(x_2,x_3,\dots,x_{r+1})\in C$ and $(x_1,x_3,\dots,x_{r+1})\in C$.
  By the definition of $C_i$, this implies that
  $(x_2,x_1,x_3,\dots,x_{r+1})\in C_1\cap C_2$. Since
  $(x_1,x_2,x_3,\dots,\allowbreak x_{r+1})\notin\sclo R$, the vector $(x_2,x_1,x_3,\dots,x_{r+1})$
  does not belong to $R$ and, therefore, belongs to $K$. This argument actually
  shows that $(x_1,x_2,x_3,\dots,x_{r+1})$ still belongs to $K\setminus\sclo R$
  after transposing any two coordinates. It follows that every permutation of
  $(x_1,x_2,x_3,\dots,x_{r+1})$ stays in~$K\setminus\sclo R$.
\end{subproof}

Claim \ref{cl:K-sym} shows that if $C$ is a covering code of density $\lambda(C)\approx1/r$
and, therefore, $\lambda(R)$ is small, then the kernel $K$ is almost symmetric.

Given $(x_1,\dots,x_{r-1},x_r)\in C$, define the \emph{splinter} of $C$ at $(x_1,\dots,x_{r-1},x_r)$
with respect to the last coordinate as the set
$$
S(x_1,\dots,x_{r-1},x_r)=\Set{x\in X}{(x_1,\dots,x_{r-1},x)\in C}.
$$
Let $\delta\in(0,1)$ be the parameter whose value will be chosen later.
Consider the part $C'$ of $C$ consisting of the vectors with small splinters. Specifically,
$$
C'=\Set{(x_1,\dots,x_r)\in C}{\lambda\of{S(x_1,\dots,x_r)}\le\delta}.
$$
Note that
\begin{equation}
  \label{eq:Cde}
\lambda(C')\le\delta.  
\end{equation}
Given $(x_1,\dots,x_{r-1},x_r)\in C$, we also define its \emph{extension-deletion set}
$E(x_1,\dots,x_{r-1},\allowbreak x_r)\subseteq X^{r+1}$ by
\begin{multline*}
E(x_1,\dots,x_{r-1},x_r)
=\bigl\{(x_1,\dots,x_{r-1},x_r,x_{r+1})\in X^{r+1}\,:\\
(x_1,\dots,x_{r-1},x_{r+1})\in C\bigr\}.
\end{multline*}
Finally, let
$$
W=\Set{(x_1,\dots,x_r)\in C\setminus C'}{E(x_1,\dots,x_r)\subseteq\sclo R}.
$$

\begin{claim}\label{cl:W}
  $\lambda(\sclo W)\le \lambda(\sclo R)/\delta$.
\end{claim}

\begin{subproof}
  Consider the set 
  $$W^+=\bigcup_{(x_1,\dots,x_r)\in W}E(x_1,\dots,x_r).$$
  Since $W^+\subseteq\sclo R$, we have $\sclo{W^+}\subseteq\sclo R$ and,
  therefore,
  \begin{equation}
    \label{eq:W1}
\lambda(\sclo{W^+})\le\lambda(\sclo R).    
  \end{equation}
  On the other hand,
  \begin{equation}
    \label{eq:W2}
\lambda(\sclo{W^+})\ge\lambda(\sclo W)\cdot\delta.
  \end{equation}
  Indeed, for $(x_1,\dots,x_r)\in\sclo  W$ let $\sigma$ be the lexicographically smallest
  permutation of $\{1,\dots,r\}$ such that $(x_{\sigma(1)},\dots,x_{\sigma(r)})\in W$.
  For every $x\in S(x_{\sigma(1)},\dots,x_{\sigma(r)})$, we have
  $(x_{\sigma(1)},\dots,x_{\sigma(r)},x)\in E(x_{\sigma(1)},\dots,x_{\sigma(r)})\subseteq W^+$
  and, hence, $(x_1,\dots,x_r,x)\in\sclo{W^+}$. To obtain Inequality \refeq{W2}, it suffices
  to note that $\lambda(S(x_{\sigma(1)},\dots,x_{\sigma(r)}))>\delta$ because $(x_{\sigma(1)},\dots,x_{\sigma(r)})\notin C'$.

  The claim readily follows from Inequalities \refeq{W1} and~\refeq{W2}.
\end{subproof}

We now show that if $\delta$ is chosen so that $\lambda(C')$ and $\lambda(W)$ are small,
then $C$ is almost symmetric.

\begin{claim}\label{cl:C-sym}
  $\sclo{C\setminus(C'\cup W)}\subseteq C$.
\end{claim}

\begin{subproof}
  For any $(x_1,\dots,x_r)\in C$, note that its extension $(x_1,\dots,x_r,x_{r+1})$
  belongs to $E(x_1,\dots,x_r)$ if and only if it belongs to $C_r\cap C_{r+1}$.
  Suppose that $(x_1,\dots,x_r)\in C\setminus(C'\cup W)$.
  By the definition of $W$, we have $E(x_1,\dots,x_r)\not\subseteq\sclo R$.
  This means that there exists $x_{r+1}\in X$ such that $(x_1,\dots,x_r,x_{r+1})$
  belongs to $C_r\cap C_{r+1}$ but not to $\sclo R$. It follows that
  $$
  (x_1,\dots,x_r,x_{r+1})\in(C_r\cap C_{r+1})\setminus \sclo R\subseteq (C_r\cap C_{r+1})\setminus R\subseteq K.
  $$
  We conclude that $(x_1,\dots,x_r,x_{r+1})\in K\setminus\sclo R$.
  Let $\sigma$ be an arbitrary permutation of $\{1,\dots,r\}$.
  By Claim \ref{cl:K-sym}, we have $(x_{\sigma(1)},\dots,x_{\sigma(r)},x_{r+1})\in K$.
  Since $K\subseteq C_{r+1}$, this implies that $(x_{\sigma(1)},\dots,x_{\sigma(r)})\in C$.
\end{subproof}

Since $C$ covers $X^{r+1}$, its symmetrization $\sclo C$ covers $X^{r+1}$ as well and,
by Claim \ref{cl:C-sym}, we have
$$
  s^\star([0,1],r+1,r)\le\lambda(\sclo C)\le\lambda(C)+\lambda(\sclo{C'})+\lambda(\sclo{W}).
$$
Taking into account Bound \refeq{Cde} and Claim \ref{cl:W}, we obtain
\begin{multline*}
  s^\star([0,1],r+1,r)\le\lambda(C)+r!\,\lambda(C')+\lambda(\sclo R)/\delta\\
  \le\lambda(C)+r!\,\delta+r!\,\lambda(R)/\delta.
\end{multline*}
Setting $\delta=\sqrt{\lambda(R)}$, we conclude that
$$
  s^\star([0,1],r+1,r)\le\lambda(C)+2r!\sqrt{\lambda(R)}.
$$
Along with Bound \refeq{lambdaR}, this implies that
\begin{multline*}
 s^\star([0,1],r+1,r)\le\lambda(C)\\
 +2r!\sqrt{r(r+1)(1-\lambda(C))\of{\lambda(C)-\frac1r}}.
\end{multline*}
Since $\lambda(C)$ can be taken arbitrarily close to $s([0,1],r+1,r)$, the proof is completed
by applying Theorems \ref{thm:s-limit} and~\ref{thm:t-limit}.
\end{proof}

\begin{corollary}\label{cor:s-t-lower}
  $s(r+1,r)\ge\frac1r+(1-o(1))\of{\frac{t(r+1,r)-1/r}{2r\cdot r!} }^2$. 
\end{corollary}

\begin{proof}
  Let $s_r=s(r+1,r)-1/r$ and $t_r=t(r+1,r)-1/r$. Set $R=r\cdot r!$. 
Using the lower bound $s(r+1,r)\ge1/r$ of Theorem \ref{thm:1/r}, 
from Theorem \ref{thm:tr-sr} we derive
$$
t_r\le s_r+2r!\sqrt{r(r+1)(1-1/r)s_r}<s_r+2R\sqrt{s_r}.
$$
This readily implies
$$
\sqrt{s_r}>\sqrt{t_r+R^2}-R=\frac{t_r}{R+\sqrt{R^2+t_r}}>\frac{t_r}{R+\sqrt{R^2+1}},
$$
yielding the desired bound.
\end{proof}

Plugging in Bound \refeq{r-odd}, we obtain the following.

\begin{corollary}\label{cor:odd}
  For odd $r$,
  $$
s(r+1,r)\ge\frac1r+\frac{1-o(1)}{4r^6(r!)^2}.
  $$
\end{corollary}

\noindent
An analog of Corollary \ref{cor:odd} for even $r$ follows from Corollary \ref{cor:s-t-lower}
by using Bound~\refeq{r-even}.

Combining Theorem \ref{thm:tr-sr} for $r=3$ with the lower bound in \refeq{r=3}, we get
$$
0.438334\le s(4,3)+24\sqrt{(1-s(4,3))(3s(4,3)-1)}.
$$
This allows us to slightly improve the lower bound $s(4,3)\ge1/3$ given by Theorem~\ref{thm:1/r}.

\begin{corollary}\label{cor:s43}
  $s(4,3)\ge0.3333429$.
\end{corollary}

\section*{Acknowledgment}
We thank the anonymous reviewers for careful reading the manuscript and useful comments.



\appendix

\normalsize

\setlength{\parindent}{4mm}

We are grateful to Chengfei Xie for pointing out that the tree $T$ in the proof of Lemma~\ref{lem:1/r} cannot be chosen arbitrarily. The argument is valid when $T$ is a path with vertices $1,2,\ldots,r+1$, which suffices to prove Theorem \ref{thm:1/r}. More specifically, our proof of the lower bound $\lambda(C_{1,2})\ge\lambda(C)^2$ extends directly to show that $\lambda(C_{i,i+1})\ge\lambda(C)^2$ for every $i\le r$. However, this bound need not hold for $C_{i,j}$ when $i$ and $j$ are not consecutive.

A similar issue occurs in the proof of Theorem \ref{thm:tr-sr}. Nevertheless, we have since found a completely different method that proves a considerably stronger result than Theorem~\ref{thm:tr-sr}.

\begin{theorem}\label{thm:s-t}
  $s(k,r)=t(k,r)$.
\end{theorem}

\begin{proof}
The inequality $s(k,r)\le t(k,r)$ holds by Corollary \ref{cor:s-t}, and
we only have to prove the inequality $t(k,r)\le s(k,r)$.

Given an (arbitrary) strict linear order $\prec$ on $[n]$ and a set $C\subseteq[n]^r$,
define a family $G_\prec$ of $r$-element subsets of $[n]$ by
\begin{multline*}
  \{x_1,\ldots,x_r\}\in G_\prec\text{ if }(x_1,\ldots,x_r)\in C\\
  \text{and }x_1\prec x_2\prec\dots\prec x_r.  
\end{multline*}

\begin{claim}\label{cl:C-G}
  If $C$ is a covering $(k-r)$-insertion code over $[n]$, then, for every strict linear order $\prec$ on $[n]$, $G_\prec$ is a Tur\'an $(n,k,r)$-system.
\end{claim}

\begin{subproof}
Consider an arbitrary $k$-element subset $\{x_1,\ldots,x_k\}$ of $[n]$.
Without loss of generality, assume that $x_1\prec x_2\prec\dots\prec x_k$.
Since $C$ is a covering $(k-r)$-insertion code, the sequence
$(x_1,\ldots,x_k)$ is covered by some subsequence $(x_i)_{i\in R}$ in $C$,
where $R\subseteq\{1,\ldots,k\}$ and $|R|=r$.
The set $\{x_1,\ldots,x_k\}$ is obviously covered by the $r$-element set
$\Set{x_i}_{i\in R}$. It remains to notice that the last set belongs to~$G_\prec$.
\end{subproof}

Recall that $C^\dagger$ denotes the set of all sequences in $C\subseteq[n]^r$ with pairwise
distinct elements. Let $C$ be fixed and $\prec$ be taken uniformly at random among all strict linear orders on $[n]$.
For a sequence $X=(x_1,\ldots,x_r)$ in $C^\dagger$, let $\xi_X$ be the indicator random variable
for the event $x_1\prec x_2\prec\dots\prec x_r$. By linearity of expectation,
\begin{equation}
  \label{eq:lin-exp}
\mean|G_\prec|=\mean\of{\sum_{X\in C^\dagger}\xi_X}=\sum_{X\in C^\dagger}\mean\,\xi_X\le\frac{|C|}{r!}.  
\end{equation}

Choose $C$ to be an optimal $(k-r)$-insertion code, i.e., $|C|=S(n,k,r)$.
Fix a strict linear order $\ll$ on $[n]$ such that $|G_\ll|\le\mean|G_\prec|$.
Using \refeq{lin-exp} and Claim \ref{cl:C-G}, we conclude that
$$
T(n,k,r)\le|G_\ll|\le\frac{S(n,k,r)}{r!}.
$$
It follows that
\begin{multline*} 
  \frac{T(n,k,r)}{{n\choose r}}\le\frac{S(n,k,r)}{n(n-1)\cdots(n-r+1)}\\
  =\frac{S(n,k,r)}{n^r}\Big/
\of{(1-\frac1n)\cdots(1-\frac{r-1}n)}.
\end{multline*}
Letting $n$ go to the infinity, we obtain the required inequality $t(k,r)\le s(k,r)$.
\end{proof}

Together with Theorems \ref{thm:s-limit} and \ref{thm:t-limit}, Theorem \ref{thm:s-t}
implies that $s([0,1],k,r)=s^\star([0,1],k,r)$. In other words, the optimal value $s([0,1],k,r)$
can be approached by symmetric covering $(k-r)$-insertion codes $C\subseteq[0,1]^r$.

Theorem \ref{thm:s-t} also strengthens Corollaries \ref{cor:odd} and \ref{cor:s43}.
More precisely, the lower bounds for $t(r+1,r)$ given in \refeq{r-odd} and \refeq{r-even} (proved in \cite{ChungL99,LuZ09})
also hold for $s(r+1,r)$. Moreover, the lower bound in \refeq{r=3}, due to Razborov \cite{Razborov10}, yields $s(4,3)\ge0.438334$.

\end{document}